\batchmode
\makeatletter
\def\input@path{{\string"/Users/paranoia/Documents/Research/mypapers/EKR theorems for simplicial complexes/\string"/}}
\makeatother
\documentclass[12pt,oneside]{amsart}
\usepackage[T1]{fontenc}
\usepackage[latin9]{inputenc}
\usepackage{verbatim}
\usepackage{amsthm}
\usepackage{amssymb}
\usepackage[unicode=true, pdfusetitle,
 bookmarks=true,bookmarksnumbered=false,bookmarksopen=false,
 breaklinks=false,pdfborder={0 0 1},backref=false,colorlinks=false]
 {hyperref}

\makeatletter
\numberwithin{equation}{section}
\numberwithin{figure}{section}
\theoremstyle{plain}
\newtheorem{thm}{Theorem}[section]
  \theoremstyle{remark}
  \newtheorem{note}[thm]{Note}
  \theoremstyle{definition}
  \newtheorem{defn}[thm]{Definition}
  \theoremstyle{plain}
  \newtheorem{conjecture}[thm]{Conjecture}
  \theoremstyle{remark}
  \newtheorem{rem}[thm]{Remark}
  \theoremstyle{plain}
  \newtheorem{lem}[thm]{Lemma}
 \theoremstyle{definition}
  \newtheorem{example}[thm]{Example}
  \theoremstyle{plain}
  \newtheorem{cor}[thm]{Corollary}
  \theoremstyle{plain}
  \newtheorem{prop}[thm]{Proposition}

\usepackage{ae}
\usepackage{aecompl}

\makeatother

\begin{document}
%
{}

\global\long\def\normalin{\mathrel{\lhd}}

\global\long\def\innormal{\mathrel{\rhd}}

\global\long\def\semidirect{\mathbin{\rtimes}}

\global\long\def\Stab{\operatorname{Stab}}

%
{}

\global\long\def\bdry{\partial}

\global\long\def\susp{\operatorname{susp}}

%
{}

\global\long\def\lrprod{\mathop{\check{\prod}}}

\global\long\def\lrtimes{\mathbin{\check{\times}}}

\global\long\def\urtimes{\mathbin{\hat{\times}}}

\global\long\def\urprod{\mathop{\hat{\prod}}}

\global\long\def\subsetdot{\mathrel{\subset\!\!\!\!{\cdot}\,}}

\global\long\def\dotsupset{\mathrel{\supset\!\!\!\!\!\cdot\,\,}}

\global\long\def\precdot{\mathrel{\prec\!\!\!\cdot\,}}

\global\long\def\dotsucc{\mathrel{\cdot\!\!\!\succ}}

\global\long\def\des{\operatorname{des}}

\global\long\def\rank{\operatorname{rank}}

\global\long\def\height{\operatorname{height}}

%
{}

\global\long\def\modreln{\mathrel{M}}

%
{}

\global\long\def\link{\operatorname{link}}

\global\long\def\freejoin{\mathbin{\circledast}}

\global\long\def\stellarsd{\operatorname{stellar}}

\global\long\def\conv{\operatorname{conv}}

\global\long\def\disjointunion{\mathbin{\dot{\cup}}}

\global\long\def\skel{\operatorname{skel}}

\global\long\def\depth{\operatorname{depth}}

\global\long\def\st{\operatorname{star}}

\global\long\def\alexdual#1{#1^{\vee}}

\global\long\def\reg{\operatorname{reg}}

\global\long\def\shift{\operatorname{Shift}}

%
{}

\global\long\def\Dom{\operatorname{Dom}}

%
{}

\global\long\def\cosetposet{\overline{\mathfrak{C}}}

\global\long\def\cosetlat{\mathfrak{C}}

\title{Erd\H{o}s-Ko-Rado theorems for simplicial complexes}

\author{Russ Woodroofe}

\email{russw@math.wustl.edu}


\address{Department of Mathematics, Washington University in St.~Louis, St.~Louis,
MO, 63130}
\begin{abstract}
A recent framework for generalizing the Erd\H{o}s-Ko-Rado Theorem,
due to Holroyd, Spencer, and Talbot, defines the Erd\H{o}s-Ko-Rado
property for a graph in terms of the graph's independent sets. Since
the family of all independent sets of a graph forms a simplicial complex,
it is natural to further generalize the Erd\H{o}s-Ko-Rado property
to an arbitrary simplicial complex. An advantage of working in simplicial
complexes is the availability of algebraic shifting, a powerful shifting
(compression) technique, which we use to verify a conjecture of Holroyd
and Talbot in the case of sequentially Cohen-Macaulay near-cones.
\end{abstract}
\maketitle

\section{\label{sec:Introduction}Introduction}

A family $\mathcal{A}$ of sets is \emph{intersecting} if every pair
of sets in $\mathcal{A}$ has non-empty intersection, and is an \emph{$r$-family}
if every set in $\mathcal{A}$ has cardinality $r$. A well-known
theorem of Erd\H{o}s, Ko, and Rado bounds the cardinality of an intersecting
$r$-family:
\begin{thm}
\emph{\label{thm:ErdosKoRadoV1} (Erd\H{o}s-Ko-Rado \cite{Erdos/Ko/Rado:1961})}
Let $r\leq\frac{n}{2}$ and $\mathcal{A}$ be an intersecting $r$-family
of subsets of $[n]$. Then $\vert\mathcal{A}\vert\leq{n-1 \choose r-1}$.
\end{thm}
Given a simplicial complex $\Delta$ (defined in Section \ref{sec:Shifting})
and a face $\sigma$ of $\Delta$, we define the \emph{link} of $\sigma$
in $\Delta$ to be \[
\link_{\Delta}\sigma=\{\tau\,:\,\tau\cup\sigma\mbox{ is a face of }\Delta,\,\tau\cap\sigma=\emptyset\}.\]
An \emph{$r$-face }of $\Delta$ is a face of cardinality $r$. We
further let $f_{r}(\Delta)$ be defined as the number of $r$-faces
in $\Delta$, and the tuple $(f_{0}(\Delta),f_{1}(\Delta),\dots,f_{d+1}(\Delta))$
(where $d$ is the dimension of $\Delta$) is called the \emph{$f$-vector}
of $\Delta$. 
\begin{note}
We follow Swartz \cite{Swartz:2006} in our definition of $r$-face
and $f_{r}$. Other sources define an $r$-face to be a face with
dimension $r$ (rather than cardinality $r$) which shifts the indices
of the $f$-vector by 1.
\end{note}
We restate Theorem \ref{thm:ErdosKoRadoV1} using this language:
\begin{thm}
\label{thm:ErdosKoRadoV2} Let $r\leq\frac{n}{2}$ and $\mathcal{A}$
be an intersecting $r$-family of faces of the simplex with $n$ vertices.
Then $\vert\mathcal{A}\vert\leq f_{r-1}(\link_{\Delta}v_{1})$.
\end{thm}
Let $G$ be a graph with edge set $E(G)$ and vertex set $V(G)$.
The \emph{independence complex} of $G$, denoted $I(G)$, is the simplicial
complex consisting of all independent sets of $G$. For a vertex $v\in V(G)$,
the \emph{closed neighborhood} $N[v]$ consists of $v$ and all its
neighbors. We notice that $\link_{I(G)}v=I(G\setminus N[v])$.

Following Holroyd, Spencer, and Talbot \cite[Section 1]{Holroyd/Spencer/Talbot:2005},
we define:
\begin{defn}
\label{def:r-EKR} A simplicial complex $\Delta$ is \emph{$r$-EKR
}if every intersecting $r$-family $\mathcal{A}$ of faces of $\Delta$
satisfies $\vert\mathcal{A}\vert\leq\max_{v\in V(\Delta)}f_{r-1}(\link_{\Delta}v)$.
Equivalently, $\Delta$ is $r$-EKR if the set of all $r$-faces containing
some $v$ has maximal cardinality among all intersecting families
of $r$-faces.
\end{defn}
Holroyd and Talbot \cite{Holroyd/Talbot:2005} further investigated
the problem of when graphs are $r$-EKR, and made the following conjecture:
\begin{conjecture}
\emph{\label{con:HolroydTalbot} (Holroyd-Talbot }\cite[Conjecture 7]{Holroyd/Talbot:2005}\emph{)
}If $G$ is a graph where the minimal facet cardinality of $I(G)$
is $k$, then $I(G)$ is $r$-EKR for $r\leq\frac{k}{2}$. 
\end{conjecture}
The natural extension was made by Borg \cite{Borg:2009}:
\begin{conjecture}
\label{con:HolyroydTalbotForComplexes}\emph{(Borg }\cite[Conjecture 1.7]{Borg:2009}\emph{)}
If $\Delta$ is a simplicial complex having minimal facet cardinality
$k$, then $\Delta$ is $r$-EKR for $r\leq\frac{k}{2}$.\end{conjecture}
\begin{rem}
A different version of an EKR property for simplicial complexes was
studied by Chvátal \cite{Chvatal:1974}, who conjectured that if $\mathcal{A}$
is an intersecting family of faces (of possibly differing dimensions)
then \[
\vert\mathcal{A}\vert\leq\max_{v\in V(\Delta)}\left(\sum_{r}f_{r}(\link_{\Delta}v)\right),\]
i.e., that the set of all faces containing some $v$ has maximal cardinality
among all intersecting families of faces. We notice that Conjecture
\ref{con:HolyroydTalbotForComplexes} is an analogue of Chvátal's
Conjecture for uniform intersecting families. 
\end{rem}
We refer the reader to \cite{Borg/Holroyd:2009} for additional background
on the $r$-EKR property in graphs, and to \cite{Borg:2009} for further
relationships with more general intersection problems.

\medskip{}

This paper is organized as follows. In Section \ref{sec:Shifting}
we review the necessary background on shifted complexes, algebraic
shifting, the Cohen-Macaulay property, and near-cones. We also characterize
the graphs $G$ such that $\shift I(G)$ is the independence complex
of some other graph, and the graphs such that $I(G)$ is a near-cone.
In Section \ref{sec:Main-theorem} we present and prove our main theorem,
Theorem \ref{thm:HolyroydTalbotDepthK}. In Section \ref{sec:Applications}
we give applications of Theorem \ref{thm:HolyroydTalbotDepthK} to
Conjecture \ref{con:HolroydTalbot}. In particular, we recover the
main result of \cite{Hurlbert/Kamat:2009UNP}, and many of the results
of \cite{Holroyd/Spencer/Talbot:2005}. We close in Section \ref{sec:Further-questions}
with further questions regarding the strict $r$-EKR property.

\section*{Acknowledgements}

The article at Gil Kalai's blog (at http://gilkalai.wordpress.com/~)
on the use of algebraic shifting in intersection theorems was very
helpful, especially in proving Lemma \ref{lem:HolyroydTalbotShifted}.
Art Duval and Isabella Novik also answered some of my questions about
algebraic shifting. Glenn Hurlbert and Vikram Kamat explained the
difficulties in extending their proof of Corollary \ref{cor:HurlbertKamat-Chordal}
to all vertex decomposable graphs.

\section{\label{sec:Shifting}Shifting}

We will need some basic simplicial complex language: An \emph{(abstract)
simplicial complex} $\Delta$ is a system of sets (called \emph{faces})
on base set $V(\Delta)$ (called \emph{vertices}) such that if $\sigma$
is a face then every subset of $\sigma$ is also a face. We assume
that every vertex is contained in some face. A \emph{facet} of $\Delta$
is a face that is maximal under inclusion. It is well-known that any
abstract simplicial complex has a \emph{geometric realization}, a
geometric simplicial complex with the same face incidences; so we
can use terms from geometry such as dimension to describe a simplicial
complex.

If $\mathcal{F}$ is some family of sets, then the simplicial complex
$\Delta(\mathcal{F})$ \emph{generated by} $\mathcal{F}$ has faces
consisting of all subsets of all sets in $\mathcal{F}$. For a simplicial
complex $\Delta$, the \emph{$r$-skeleton} $\Delta^{(r)}$ consists
of all faces of $\Delta$ having dimension at most $r$, while the
\emph{pure $r$-skeleton} is the subcomplex generated by all faces
of $\Delta$ having dimension exactly $r$. The \emph{join} of disjoint
simplicial complexes $\Delta$ and $\Sigma$ is the simplicial complex
$\Delta*\Sigma$ with faces $\tau\cup\sigma$, where $\tau$ is a
face of $\Delta$ and $\sigma$ a face of $\Sigma$.

A simplicial complex $\Delta$ with ordered vertex set $\{v_{1},\dots,v_{n}\}$
is \emph{shifted} if whenever $\sigma$ is a face of $\Delta$ containing
vertex $v_{i}$, then $\left(\sigma\setminus\{v_{i}\}\right)\cup\{v_{j}\}$
is a face of $\Delta$ for every $j<i$. An $r$-family $\mathcal{F}$
of subsets of $\{v_{1},\dots,v_{n}\}$ is \emph{shifted} if it generates
a shifted complex.

A general approach to proving theorems similar to Theorem \ref{thm:ErdosKoRadoV2}
is to define a \emph{shifting operation} or \emph{compression operation}
which replaces a non-shifted set system with a shifted system obeying
some of the same combinatorial properties. Erd\H{o}s, Ko, and Rado
pioneered this technique in \cite{Erdos/Ko/Rado:1961}, and their
operation is now called combinatorial shifting. Combinatorial shifting
is discussed in the survey article \cite{Frankl:1987}, particularly
as applied to intersection theorems.

\subsection{Algebraic shifting}

The specific shifting operation we will use is called \emph{exterior
algebraic shifting} \emph{(with respect to a field $F$)}, and we
denote the (exterior) algebraic shift of $\Delta$ by $\shift\Delta$,
or $\shift_{F}\Delta$ if we want to emphasize the field. Algebraic
shifting was first studied by Kalai, and unless otherwise stated the
facts we present here were first proved by him. 

The precise definition of $\shift\Delta$ will not be important for
us, but can be found in Kalai's survey article \cite{Kalai:2002}.
Rather than working with the definition, we examine $\shift\Delta$
using a series of lemmas collected in \cite{Kalai:2002}. The following
basic properties we will use without further mention:
\begin{lem}
\label{lem:BasicFacts}\cite{Kalai:2002} Let $\Delta$ be a simplicial
complex with $n$ vertices. Then:
\begin{enumerate}
\item $\shift\Delta$ is a shifted simplicial complex on an ordered vertex
set $\{v_{1},\dots,v_{n}\}$.
\item If $\Delta$ is shifted, then $\shift\Delta\cong\Delta$.
\item $f_{i}(\shift\Delta)=f_{i}(\Delta)$ for all $i$.
\end{enumerate}
\end{lem}
Shifting respects subcomplexes, at least in a weak sense:
\begin{lem}
\cite[Theorem 2.2]{Kalai:2002}\label{lem:ShiftedContainment} If
$\Sigma\subset\Delta$ are simplicial complexes, then $\shift\Sigma\subset\shift\Delta$.\end{lem}
\begin{example}
Let $\Delta$ be the simplicial complex with facets $\{1,2\}$ and
$\{3,4\}$. Then it is easy to see that the unique shifted complex
with the same $f$-vector has facets $\{1,2\}$, $\{2,3\}$, and $\{4\}$,
hence that this complex is $\shift\Delta$. Then $\shift\{1,2\}=\shift\{3,4\}=\{1,2\}\subset\shift\Delta$. \end{example}
\begin{cor}
If $\Delta$ is a simplicial complex, then $\shift(\Delta^{(r)})=(\shift\Delta)^{(r)}$.\end{cor}
\begin{proof}
Lemma \ref{lem:ShiftedContainment} gives that $\shift(\Delta^{(r)})\subseteq(\shift\Delta)^{(r)}$,
and by Lemma \ref{lem:BasicFacts} part (3) the $f$-vectors are equal.
\end{proof}
Notice that if $\Delta$ is a shifted complex, then it is immediate
that $\Delta^{(r)}$ is shifted for every $r$.

If $\mathcal{A}$ is some $r$-family of sets, then $\shift\mathcal{A}$
is the pure $r$-skeleton of $\shift\Delta(\mathcal{A})$. (Kalai's
equivalent definition actually defines $\shift\Delta$ as a union
of the shift of its $r$-faces \cite[Section 2.1]{Kalai:2002}.) Kalai
proves:
\begin{lem}
\cite[Corollary 6.3]{Kalai:2002}\label{lem:ShiftInterestingFamily}
If $\mathcal{A}$ is an intersecting $r$-family, then $\shift\mathcal{A}$
is an intersecting $r$-family.
\end{lem}

\subsection{\label{sub:Near-cones}Near-cones}

A simplicial complex $\Delta$ is a \emph{near-cone} with respect
to an \emph{apex vertex $v$} if for every face $\sigma$, the set
$(\sigma\setminus\{w\})\cup\{v\}$ is also a face for each vertex
$w\in\sigma$. Equivalently, the boundary of every facet of $\Delta$
is contained in $v*\link_{\Delta}v$; another equivalent condition
is that $\Delta$ consists of $v*\link_{\Delta}v$ union some set
of facets not containing $v$ (but whose boundary is contained in
$\link_{\Delta}v$). If $\Delta$ is a cone with apex vertex $v$,
then obviously $\Delta=v*\link_{\Delta}v$, thus every cone is a near-cone.

Because the apex vertex is always the vertex with the largest link,
near-cones are relatively easy to work with in the context of intersection
theorems, as has been previously noticed in e.g. \cite{Snevily:1992,Borg:2007}.
In particular:
\begin{lem}
\label{cor:NearConeApexMaxlLink}If $\Delta$ is a near-cone with
apex vertex $v$, then $f_{r}(\link_{\Delta}w)\leq f_{r}(\link_{\Delta}v)$
for any vertex $w$ and all $r$.\end{lem}
\begin{proof}
For every $(r+1)$-face $\sigma$ containing $w$, either $v\in\sigma$
or else $(\sigma\setminus w)\cup v$ is an $r$-face containing $v$
but not $w$.
\end{proof}
Notice that $\Delta$ being a near-cone with apex vertex $v$ essentially
says that $\Delta$ is {}``shifted with respect to $v$''. In particular,
any shifted complex is a near-cone with apex vertex $v_{1}$. Nevo
examined the algebraic shift of a near-cone, showing:
\begin{lem}
\emph{(Nevo }\cite[Theorems 5.2 and 5.3]{Nevo:2005}\emph{)} \label{lem:ShiftedNearCone}
If $\Delta$ is a near-cone with apex $v$, let us consider $\shift(\link_{\Delta}v)$
as having ordered vertex set $\{v_{2},\dots,v_{n}\}$. Then 
\begin{enumerate}
\item $\link_{\shift\Delta}v_{1}=\shift(\link_{\Delta}v)$. 
\item $\shift\Delta=\left(v_{1}*\shift(\link_{\Delta}v)\right)\cup\mathcal{B}$,
where $\mathcal{B}$ is a set of facets not containing $v_{1}$. 
\end{enumerate}
\end{lem}
\begin{cor}
\label{cor:NearConePreserveFApex}If $\Delta$ is a near-cone with
apex $v$, then $f_{r}(\link_{\Delta}v)=f_{r}(\link_{\shift\Delta}v_{1})$
for all $r$.
\end{cor}

\subsection{Pure complexes, Cohen-Macaulay complexes, and depth}

A simplicial complex $\Delta$ is \emph{pure} if all facets of $\Delta$
have the same dimension. Graphs with a pure independence complex are
sometimes called \emph{well-covered}.

Let $F$ be either any field, or the ring of integers. A simplicial
complex $\Delta$ is \emph{Cohen-Macaulay over F} if its homology
satisfies $\tilde{H}_{i}(\link_{\Delta}\sigma;F)=0$ for all $i<\dim(\link_{\Delta}\sigma)$
and all faces $\sigma$ of $\Delta$ (including $\sigma=\emptyset$).
It is well-known that every Cohen-Macaulay complex is pure, and that
every skeleton of a Cohen-Macaulay complex is Cohen-Macaulay.

A simplicial complex is \emph{sequentially Cohen-Macaulay over $F$}
if the pure $r$-skeleton of $\Delta$ is Cohen-Macaulay (over $F$)
for all $r$. Thus, a pure sequentially Cohen-Macaulay complex is
Cohen-Macaulay. 

When we simply say that a simplicial complex $\Delta$ is (sequentially)
Cohen-Macaulay, with no mention of $F$, then we mean that $\Delta$
is (sequentially) Cohen-Macaulay over all $F$. For example, every
{}``shellable'' or {}``vertex decomposable'' complex is sequentially
Cohen-Macaulay over any $F$ \cite{Bjorner/Wachs:1996,Bjorner/Wachs:1997}. 

The main relationships between the Cohen-Macaulay property and shifting
are the following:
\begin{lem}
\cite[Theorem 11.3]{Bjorner/Wachs:1997} If $\Delta$ is shifted,
then $\Delta$ is {}``vertex decomposable'', hence sequentially
Cohen-Macaulay (over any $F$).
\end{lem}

\begin{lem}
\label{lem:ShiftOfCMisPure}\emph{ (\cite[Theorem 4.1]{Kalai:2002},
see also \cite[Proposition 8.4]{Aramova/Herzog:2000})} $\shift_{F}\Delta$
is pure if and only if $\Delta$ is Cohen-Macaulay (over $F$).
\end{lem}
Duval \cite{Duval:1996} also examined the algebraic shift of a sequentially
Cohen-Macaulay complex, and more generally of Cohen-Macaulay skeletons.
A result of his that will be of particular interest to us is:
\begin{cor}
\label{lem:ShiftOfCMSkel}\emph{(Duval \cite[Corollary 4.5]{Duval:1996})}
The minimum facet dimension of $\shift_{F}\Delta$ is $\geq d$ if
and only if $\Delta^{(d)}$ is Cohen-Macaulay (over $F$).
\end{cor}
Corollary \ref{lem:ShiftOfCMSkel} suggests the definition of the
\emph{depth of $\Delta$ over $F$} as\[
\depth_{F}\Delta=\max\{d\,:\,\Delta^{(d)}\mbox{ is Cohen-Macaulay over }F\}.\]
Thus, $\depth_{F}\Delta$ is the minimum facet dimension of $\shift_{F}\Delta$.
We note that $\depth_{F}\Delta$ is one less than the ring-theoretic
depth of the {}``Stanley-Reisner ring''\emph{ $F[\Delta]$} \cite[Theorem 3.7]{Smith:1990}.
Thus, just as in the ring-theoretic situation, $\Delta$ is Cohen-Macaulay
over $F$ if and only if $\depth_{F}\Delta=\dim\Delta$. If $\Delta$
is sequentially Cohen-Macaulay over $F$ then $\depth_{F}\Delta$
is the minimum facet dimension of $\Delta$.

By the definition of simplicial homology we have $\tilde{H}_{d}(\Delta^{(d+1)};F)=\tilde{H}_{d}(\Delta;F)$,
hence the easy equivalent characterization: \[
\depth_{F}\Delta=\max\{d\,:\,\tilde{H}_{i}(\link_{\Delta}\sigma;F)=0\mbox{ for all }\sigma\in\Delta\mbox{ and }i<d-\vert\sigma\vert\}.\]
In particular, we notice that $\depth_{F}\Delta$ is at most the minimal
facet dimension, since if $\sigma$ is a facet then $\tilde{H}_{-1}(\link_{\Delta}\sigma;F)=\tilde{H}_{-1}(\emptyset;F)=F$. 

The following result about the depth of the join of complexes will
be especially useful in Section \ref{sec:Applications}:
\begin{lem}
\label{lem:DepthOfJoin} Let $F$ be a field, and $\Delta_{1}$ and
$\Delta_{2}$ be simplicial complexes. Then $\depth_{F}(\Delta_{1}*\Delta_{2})=\depth_{F}\Delta_{1}+\depth_{F}\Delta_{2}+1$.\end{lem}
\begin{proof}
Faces of $\Delta_{1}*\Delta_{2}$ have the form $\sigma=\sigma_{1}\disjointunion\sigma_{2}$
where $\sigma_{i}$ is a face of $\Delta_{i}$, hence $\link_{\Delta_{1}*\Delta_{2}}\sigma=\link_{\Delta_{1}}\sigma_{1}*\link_{\Delta_{2}}\sigma_{2}$.
The result then follows from the standard algebraic topology fact
\cite[Corollary 4.23]{Jonsson:2008} that \linebreak{}
$\tilde{H}_{n+1}(\Delta_{1}*\Delta_{2};F)=\bigoplus_{i=-1}^{n+1}\tilde{H}_{i}(\Delta_{1};F)\otimes\tilde{H}_{n-i}(\Delta_{2};F)$.
\end{proof}
The reader is referred to \cite{Jonsson:2008} for additional background
on $\depth_{F}\Delta$ and the (sequentially) Cohen-Macaulay property.
We will henceforth take the field $F$ to be understood, and drop
it from our notation.

\subsection{Shifting independence complexes}

A graph $G$ is \emph{chordal} if every induced subgraph of $G$ which
is a cyclic graph has length 3. A graph is \emph{co-chordal} if its
complement graph is chordal.

Since the original question of Holroyd, Spencer and Talbot was restricted
to the independence complexes of graphs, one might ask when $\shift I(G)$
is isomorphic to $I(G')$ for some graph $G'$. The answer is easy,
given the necessary machinery. The \emph{Alexander dual} of $\Delta$,
denoted $\alexdual{\Delta}$, is the complex with facets $\{\sigma\,:\, V\setminus\sigma\mbox{ is a minimal non-face of }\Delta\}$.
See e.g. \cite[Section 6]{Kalai:1983} for more information and background
on Alexander duality.
\begin{thm}
Let $G$ be a graph. Then $\shift I(G)$ is the independence complex
$I(G')$ of some other graph $G'$ if and only if $G$ is co-chordal.\end{thm}
\begin{proof}
We need the following three facts about Alexander duality: 1) It is
clear from the definition that $\Delta$ is the independence complex
of a graph if and only if $\alexdual{\Delta}$ is pure $(n-2)$-dimensional,
where $n$ is the number of vertices of $\Delta$. 2) Alexander duality
and shifting commute, i.e. $\shift\Delta=\alexdual{\left(\shift\left(\alexdual{\Delta}\right)\right)}$
\cite[Section 3.5.6]{Kalai:2002}. 3) If $G$ is a graph, then $\alexdual{I(G)}$
is Cohen-Macaulay if and only if $G$ is co-chordal \cite[Proposition 8]{Eagon/Reiner:1998}.
The result is then immediate from Lemma \ref{lem:ShiftOfCMisPure}.\end{proof}
\begin{rem}
The family of independence complexes of graphs has been extensively
studied in the literature under the name of \emph{flag complexes}.
\end{rem}
The shifted flag complexes were classified by Klivans, as follows:
Given a graph $G$, let $D(G)$ be $G\disjointunion\{v\}$ for a new
vertex $v$ ($D$ for {}``disjoint union''). Let $S(G)$ be the
graph on vertex set $V(G)\cup\{v\}$ for a new vertex $v$ and with
edge set $E(G)\cup\{wv\,:\, w\in V(G)\}$ ($S$ for {}``star'').
In the independence complex, we have $I(D(G))$ as the cone over $I(G)$,
and $I(S(G))$ as $I(G)\disjointunion\{v\}$, thus if $G$ is sequentially
Cohen-Macaulay then both $D(G)$ and $S(G)$ are. 

A graph is \emph{threshold} \cite{Mahadev/Peled:1995} if it is obtained
from a single vertex by some sequence of $D$ and $S$ operations.
Every threshold graph is both chordal and co-chordal. Since a $D$
operation adds a cone vertex, which can be taken as the initial vertex
in a shifted complex; and an $S$ operation adds a disjoint vertex,
which can be taken as the final vertex in a shifted complex, we have
proved inductively that the independence complex of any threshold
graph is shifted. Klivans \cite{Klivans:2007} showed the converse
result that all graphs with shifted independence complex are threshold.

We prove the following generalization of \cite[Theorem 1]{Klivans:2007}:
\begin{prop}
\label{pro:FlagNearCone} If $G$ is a graph such that $I(G)$ is
a near-cone, then $G$ is obtained from some graph $G_{0}$ by a sequence
of $D$ and $S$ operations, including at least one $D$ operation.\end{prop}
\begin{proof}
Let $v$ be the apex vertex of $I(G)$, and suppose that $wv\in E(G)$.
Then $wx$ is also an edge for every $x\in V(G)$, since if $wx$
were a face of $I(G)$ then $wv$ would also be independent, a contradiction.
We see that $G=S^{k}D(G\setminus N[v])$, where $k$ is the number
of neighbors of $v$.\end{proof}
\begin{rem}
\label{rem:NontrivialNearConesHaveIsolated} Since the independence
complex of $S(G)$ has an isolated vertex, its minimum facet dimension
is 0. Hence Proposition \ref{pro:FlagNearCone} tells us that for
a graph $G$ we have that $I(G)$ is a near-cone with non-trivial
minimum facet dimension if and only if $G$ has an isolated vertex.
\end{rem}

\section{\label{sec:Main-theorem}Main theorem}

The following lemma follows from Borg's more general result \cite[Theorem 2.7]{Borg:2009}.
We use algebraic shifting to give a short new proof of the specific
result.
\begin{lem}
\label{lem:HolyroydTalbotShifted}\emph{(Borg \cite[Theorem 2.7]{Borg:2009})}
If $\Delta$ is a shifted complex having minimal facet cardinality
$k$, then $\Delta$ is $r$-EKR for $r\leq\frac{k}{2}$.\end{lem}
\begin{proof}
Let $\Delta$ have ordered vertex set $\{v_{1},\dots,v_{n}\}$, and
let $\mathcal{A}$ be an intersecting $r$-family of faces of $\Delta$.
We proceed by induction: our base cases are when $\Delta$ is a simplex
(Theorem \ref{thm:ErdosKoRadoV2}), and the trivial case where $r=1$.

If $\Delta$ is not a simplex and $r>1$, then by Lemmas \ref{lem:ShiftInterestingFamily}
and \ref{lem:ShiftedContainment}, we have that $\shift\mathcal{A}$
is a shifted intersecting $r$-family of faces of $\Delta=\shift\Delta$
with $\vert\shift\mathcal{A}\vert=\vert\mathcal{A}\vert$. We decompose
$\shift\mathcal{A}$ into the subfamilies $\mathcal{C}$ consisting
of all $\sigma\in\shift\mathcal{A}$ with $v_{n}\in\sigma$, and $\mathcal{D}=(\shift\mathcal{A})\setminus\mathcal{C}$,
so that $\vert\mathcal{A}\vert=\vert\shift\mathcal{A}\vert=\vert\mathcal{C}\vert+\vert\mathcal{D}\vert$.

We first consider $\mathcal{C}$. Let $\mathcal{C}_{0}=\{\sigma\setminus\{v_{n}\}\,:\,\sigma\in\mathcal{C}\}$,
so that $\vert\mathcal{C}\vert=\vert\mathcal{C}_{0}\vert$. Suppose
that $\mathcal{C}_{0}$ is not intersecting. Then there are $\sigma,\tau\in\mathcal{C}$
such that $\sigma\cap\tau=\{v_{n}\}$, and $\vert\sigma\cup\tau\vert<r+r\leq k<n$.
It follows that there is a $v_{\ell}\notin\sigma\cup\tau$. But then
$\tau'=(\tau\setminus\{v_{n}\})\cup\{v_{\ell}\}$ is in $\shift\mathcal{A}$
by the definition of shiftedness, and $\sigma\cap\tau'=\emptyset$,
which contradicts that $\shift\mathcal{A}$ is intersecting. We conclude
that $\mathcal{C}_{0}$ is an intersecting $(r-1)$-family of faces
of $\link_{\Delta}v_{n}$. Since $\link_{\Delta}v_{n}$ is a shifted
complex with minimum facet cardinality at least $k-1$, we get that
$\vert\mathcal{C}\vert=\vert\mathcal{C}_{0}\vert\leq f_{r-2}(\link_{\Delta}\{v_{1},v_{n}\})$
by induction and Lemma \ref{cor:NearConeApexMaxlLink}.

We now consider $\mathcal{D}$. It is obvious that $\mathcal{D}$
is an intersecting $r$-family contained in the shifted complex $\Delta\setminus v_{n}$.
Since $\Delta$ is not a simplex, it follows easily from the definition
of shiftedness that the minimum facet cardinality of $\Delta\setminus v_{n}$
is at least $k$. By induction and Lemma \ref{cor:NearConeApexMaxlLink}
we have $\vert\mathcal{D}\vert\leq f_{r-1}(\link_{\Delta\setminus v_{n}}v_{1})$.

Putting our two parts together, we have \[
\vert\mathcal{A}\vert=\vert\mathcal{C}\vert+\vert\mathcal{D}\vert\leq f_{r-2}(\link_{\Delta}\{v_{1},v_{n}\})+f_{r-1}(\link_{\Delta\setminus v_{n}}v_{1})=f_{r-1}(\link_{\Delta}v_{1}).\qedhere\]
\end{proof}
\begin{rem}
Our requirement for Lemma \ref{lem:HolyroydTalbotShifted} on the
minimum facet cardinality seems much stronger than necessary. Our
essential need is for a parameter $k$ which we can control in both
$\Delta\setminus v_{n}$ and $\link_{\Delta}v_{n}$, and such that
$r\leq\frac{k}{2}$ forces $r<\frac{n}{2}$. Use of another such parameter
might give a stronger result version of Lemma \ref{lem:HolyroydTalbotShifted}.
Any strengthening of Lemma \ref{lem:HolyroydTalbotShifted} would
likely also strengthen Theorem \ref{thm:HolyroydTalbotDepthK}.

Holroyd and Talbot \cite[Section 3]{Holroyd/Talbot:2005} construct
several examples of independence complexes that are not $r$-EKR for
various $r$, which may give some intuition about what parameters
are tractable. (They in particular construct an example with maximum
facet cardinality $\ell$ such that $\Delta$ is not $\left\lfloor \frac{\ell}{2}\right\rfloor $-EKR.)
\end{rem}
By applying algebraic shifting to an arbitrary complex, we prove:
\begin{thm}
\label{thm:HolyroydTalbotDepthK} If $\Delta$ is a near-cone and
$F$ is an arbitrary field, then $\Delta$ is $r$-EKR for $r\leq\frac{\depth_{F}\Delta+1}{2}$.\end{thm}
\begin{proof}
Let $\mathcal{A}$ be an intersecting $r$-family of faces of $\Delta$.
By Lemma \ref{cor:NearConeApexMaxlLink}, we need to show that $\vert\mathcal{A}\vert\leq f_{r-1}(\link_{\Delta}v)$
for the apex vertex $v$. 

Apply algebraic shifting. $\shift_{F}\mathcal{A}$ is an intersecting
$r$-family of faces of $\shift\Delta$ with $\vert\shift\mathcal{A}\vert=\vert\mathcal{A}\vert$
by Lemmas \ref{lem:ShiftInterestingFamily} and \ref{lem:ShiftedContainment}.
By Lemma \ref{lem:ShiftOfCMSkel} and the following discussion, the
minimum facet cardinality of $\shift_{F}\Delta$ is $\depth_{F}\Delta+1$,
hence $\vert\mathcal{A}\vert\leq f_{r-1}(\link_{\shift\Delta}v_{1})=f_{r-1}(\link_{\Delta}v)$
by Lemma \ref{lem:HolyroydTalbotShifted} and Corollary \ref{cor:NearConePreserveFApex}.
\end{proof}
To the best of my knowledge, Theorem \ref{thm:HolyroydTalbotDepthK}
is the first `new' intersection theorem to be proved by algebraic
shifting.
\begin{cor}
\label{cor:SeqCMisrEKR} If $\Delta$ is a sequentially Cohen-Macaulay
near-cone with minimum facet cardinality $k$, then $\Delta$ is $r$-EKR
for $r\leq\frac{k}{2}$.\end{cor}
\begin{rem}
Borg's aforementioned result \cite[Theorem 2.7]{Borg:2009} generalizes
Lemma \ref{lem:HolyroydTalbotShifted} to include non-uniform families
(i.e., sets of different sizes), and to $t$-intersecting families
(i.e., to families where $\vert A\cap B\vert\geq t$). By Kalai \cite[Corollary 6.3 and following]{Kalai:2002},
algebraic shifting preserves the $t$-intersecting property for any
$r$-family, hence a reduction to \cite[Theorem 2.7]{Borg:2009} similar
to that in Theorem \ref{thm:HolyroydTalbotDepthK} will show that
if $\Delta$ is a $t$-fold near-cone (i.e., shifted with respect
to its first $t$ elements) with depth equal to its minimum facet
dimension, then \cite[Conjecture 2.7]{Borg:2009} holds for uniform
$r$-families of faces in $\Delta$. In particular, \cite[Conjecture 2.7]{Borg:2009}
holds for sequentially Cohen-Macaulay $t$-fold near-cones.
\end{rem}

\section{\label{sec:Applications}Applications}

It is immediate from the definitions that if $G=G_{1}\disjointunion G_{2}$,
then $I(G)$ decomposes as the join $I(G_{1})*I(G_{2})$. In particular,
if $G$ has an isolated vertex $v$ then $I(G)$ is a cone over $v$,
as discussed in detail in Section \ref{sub:Near-cones}. We will call
a graph $G$ \emph{sequentially Cohen-Macaulay} if its independence
complex $I(G)$ is sequentially Cohen-Macaulay. An immediate consequence
of Corollary \ref{cor:SeqCMisrEKR} is:
\begin{thm}
\label{thm:HolyroydTalbotSCMforGraphs} If $G$ is a sequentially
Cohen-Macaulay graph with an isolated vertex, then $G$ satisfies
Conjecture \ref{con:HolroydTalbot}. 
\end{thm}
The family of sequentially Cohen-Macaulay graphs includes:
\begin{enumerate}
\item Chordal graphs. \cite{Francisco/VanTuyl:2007}
\item Graphs with no induced cycle of length other than $3$ or $5$. \cite{Woodroofe:2009a}
\item Bipartite graphs containing a vertex $v$ of degree $1$ such that
$G\setminus N[v]$ and $G\setminus N[w]$ (where $w$ is the unique
neighbor of $v$) recursively satisfy the same condition. \cite[Corollary 3.11]{VanTuyl/Villarreal:2008}
\item Incomparability graphs of shellable posets. \cite{Bjorner/Wachs:1996}
\item The minimal set of non-faces of the isocahedron, or any other polytope
where the set of minimal non-faces forms a graph. \cite{Bruggesser/Mani:1971}
\item Disjoint unions of sequentially Cohen-Macaulay graphs, since $I(G_{1}\disjointunion G_{2})=I(G_{1})*I(G_{2})$.
\end{enumerate}
In particular, we recover the following theorem of Hurlbert and Kamat:
\begin{cor}
\emph{\label{cor:HurlbertKamat-Chordal} (Hurlbert-Kamat }\cite[Theorem 1.22]{Hurlbert/Kamat:2009UNP}\emph{)}
If $G$ is a chordal graph with an isolated vertex, then $G$ satisfies
Conjecture \ref{con:HolroydTalbot}.
\end{cor}
Obviously we also have that if $G$ is e.g. a threshold graph, then
$G$ satisfies Conjecture \ref{con:HolroydTalbot}. But Remark \ref{rem:NontrivialNearConesHaveIsolated}
tells us that this result is not an interesting improvement on Corollary
\ref{cor:HurlbertKamat-Chordal}, since in this case the minimum facet
cardinality of $I(G)$ is 1 unless $G$ has an isolated vertex.

\medskip{}

We apply Lemma \ref{lem:DepthOfJoin} for a result in a slightly different
direction:
\begin{prop}
\label{pro:DisjointUnionEKR} If $G=G_{1}\disjointunion\dots\disjointunion G_{n}$
is the disjoint union of $n\geq2r$ nonempty graphs, including at
least one isolated vertex, then $I(G)$ is $r$-EKR.\end{prop}
\begin{proof}
The $0$-skeleton of any non-empty complex is Cohen-Macaulay, hence
$\depth I(G_{i})\geq0$, and by repeated application of Lemma \ref{lem:DepthOfJoin}
we get $\depth I(G)=\depth I(G_{1})*\cdots*I(G_{n})\geq n-1$. The
result then follows by Theorem \ref{thm:HolyroydTalbotDepthK}.
\end{proof}
Proposition \ref{pro:DisjointUnionEKR} significantly improves \cite[Theorem 8]{Holroyd/Spencer/Talbot:2005},
which proves the result in the special case where each $G_{i}$ is
a complete graph, path, or cycle. By considering graphs of depth 1,
we can do slightly better.
\begin{lem}
\label{lem:IndComplexDepth1} The independence complex of a graph
$G$ has depth $\geq1$ if and only if $\vert G\vert>1$ and the complement
graph $\bar{G}$ is connected.\end{lem}
\begin{proof}
The complement graph $\bar{G}$ forms the $1$-skeleton of $I(G)$
under the hypothesis, and a $1$-dimensional complex is Cohen-Macaulay
if and only if it is connected.\end{proof}
\begin{example}
Let $C_{n}$ be the cyclic graph on $n$ vertices. If $n\geq5$, then
$C_{n}$ satisfies the conditions of Lemma \ref{lem:IndComplexDepth1},
hence $\depth I(C_{n})\geq1$. But the cyclic graph $C_{4}$ on $4$
vertices has disconnected complement graph, hence $\depth I(C_{4})=0$.\end{example}
\begin{prop}
\label{pro:DisjointUnionEKR-1Skel} Let $G=G_{1}\disjointunion\dots\disjointunion G_{n}$
be the disjoint union of $n$ graphs, including at least one isolated
vertex. Suppose that $m$ of the $G_{i}$ satisfy the conditions of
Lemma \ref{lem:IndComplexDepth1}. Then $I(G)$ is $r$-EKR for $r\leq\frac{n+m}{2}$.
\end{prop}
\noindent The proof is exactly as in Proposition \ref{pro:DisjointUnionEKR}.

\section{\label{sec:Further-questions}Further questions}

As we have discussed, for $\Delta$ to be $r$-EKR means that every
maximal intersecting $r$-family of faces $\mathcal{A}$ has $\vert\mathcal{A}\vert\leq\max_{v\in V(\Delta)}f_{r-1}(\link_{\Delta}v)$.
We say that $\Delta$ is \emph{strictly $r$-EKR }if every maximum
cardinality intersecting $r$-family of faces $\mathcal{A}$ consists
of the $r$-faces of $v*(\link_{\Delta}v)^{(r-1)}$ for some $v$.
That is to say, every maximum intersecting $r$-family $\mathcal{A}$
satisfies $\bigcap_{A\in\mathcal{A}}A\neq\emptyset$. Hilton and Milner
\cite{Hilton/Milner:1967} improved Theorem \ref{thm:ErdosKoRadoV2}
to:
\begin{thm}
\emph{(Hilton-Milner }\cite{Hilton/Milner:1967}\emph{)} If $\Delta$
is the simplex with $n$ vertices, then $\Delta$ is strictly $r$-EKR
for $2\leq r<\frac{n}{2}$.
\end{thm}
Holroyd and Talbot, and later Borg, actually conjectured slightly
more than we stated in Conjectures \ref{con:HolroydTalbot} and \ref{con:HolyroydTalbotForComplexes}:
\begin{conjecture}
\emph{\label{con:strictEKR} (Holroyd-Talbot \cite[Conjecture 7]{Holroyd/Talbot:2005};
Borg \cite[Conjecture 1.7]{Borg:2009})} If $\Delta$ is a simplicial
complex having minimal facet cardinality $k$, then $\Delta$ is $r$-EKR
for $r\leq\frac{k}{2}$, and strictly $r$-EKR for $r<\frac{k}{2}$.
\end{conjecture}
Can an algebraic shifting (or some other) argument be adapted to prove
Conjecture \ref{con:strictEKR}, optionally restricted to the case
of a sequentially Cohen-Macaulay complex?

\bigskip{}

Of course, even the more restricted Conjecture \ref{con:HolyroydTalbotForComplexes}
remains open for general complexes. Theorem \ref{thm:HolyroydTalbotDepthK}
suggests that a counterexample to Conjecture \ref{con:HolyroydTalbotForComplexes},
if one exists, should be a complex which badly fails to be Cohen-Macaulay.
We discuss briefly some examples of such complexes:
\begin{example}
The facets of the boundary $\Delta_{0}$ of a simplex with $n+1$
vertices is intersecting, hence not $n$-EKR. One can increase the
dimension by coning $k$ points over each facet to obtain a pure complex
$\Delta$ with $(k+1)\cdot(n+1)$ points. Since $\tilde{H}_{n}(\Delta)\neq0$,
the complex is not Cohen-Macaulay, and the hope for a counterexample
would be: for some $k>n$ and $n\leq r\leq\lfloor\frac{k+n}{2}\rfloor$,
that the family $\mathcal{A}$ consisting of all $r$-faces that contain
a facet of $\Delta_{0}$ would be larger than $f_{r-1}(\link_{\Delta}v)$. 

But we count: if $v\in\Delta_{0}$, then $f_{r-1}(\link_{\Delta}v)=n\cdot{n+k \choose r-1}$,
while $\vert\mathcal{A}\vert=(n+1)\cdot{k \choose r-n}$. A straightforward
computation (cancel, then match terms) yields that $f_{r-1}(\link_{\Delta}v)/\vert\mathcal{A}\vert>1$.
Hence $\vert\mathcal{A}\vert<f_{r-1}(\link_{\Delta}v)$, and thus
$\mathcal{A}$ and $\Delta$ are not a counterexample to Conjecture
\ref{con:HolyroydTalbotForComplexes}. 
\end{example}

\begin{example}
Cyclic graphs $C_{n}$ are not sequentially Cohen-Macaulay for $n\neq3,5$
\cite[Proposition 4.1]{Francisco/VanTuyl:2007}. For example, $I(C_{4})$
consists of two disjoint 2-faces, while $I(C_{7})$ is a triangulation
of the Möbius strip. Nonetheless, Talbot \cite{Talbot:2003} showed
that the independence complex of every cyclic graph is $r$-EKR for
all $r$. More recently, the independence complex of the disjoint
union of two cycles \cite{Hilton/Holroyd/Spencer:2010UNP}, and of
the disjoint union of an arbitrary number of cycles and a path \cite{Hilton/Spencer:2007}
were shown to be $r$-EKR for all $r$. Conjecture \ref{con:HolroydTalbot}
also holds \cite{Borg/Holroyd:2009} for the disjoint union of an
isolated vertex and a somewhat wider class of non-sequentially Cohen-Macaulay
graphs, including cycles and complete multipartite graphs. 
\end{example}
\bibliographystyle{hamsplain}
\bibliography{5_Users_paranoia_Documents_Research_Master}

\def\cprime{$'$}
\providecommand{\bysame}{\leavevmode\hbox to3em{\hrulefill}\thinspace}
\providecommand{\href}[2]{#2}
\begin{thebibliography}{10}

\bibitem{Aramova/Herzog:2000}
Annetta Aramova and J{\"u}rgen Herzog, \emph{Almost regular sequences and
  {B}etti numbers}, Amer. J. Math. \textbf{122} (2000), no.~4, 689--719.

\bibitem{Bjorner/Wachs:1996}
Anders Bj{\"o}rner and Michelle~L. Wachs, \emph{Shellable nonpure complexes and
  posets. {I}}, Trans. Amer. Math. Soc. \textbf{348} (1996), no.~4, 1299--1327.

\bibitem{Bjorner/Wachs:1997}
\bysame, \emph{Shellable nonpure complexes and posets. {II}}, Trans. Amer.
  Math. Soc. \textbf{349} (1997), no.~10, 3945--3975.

\bibitem{Borg:2007}
Peter Borg, \emph{Intersecting systems of signed sets}, Electron. J. Combin.
  \textbf{14} (2007), no.~1, Research Paper 41, 11 pp. (electronic).

\bibitem{Borg:2009}
\bysame, \emph{Extremal {$t$}-intersecting sub-families of hereditary
  families}, J. Lond. Math. Soc. (2) \textbf{79} (2009), no.~1, 167--185.

\bibitem{Borg/Holroyd:2009}
Peter Borg and Fred Holroyd, \emph{The {E}rd{\H o}s-{K}o-{R}ado properties of
  various graphs containing singletons}, Discrete Math. \textbf{309} (2009),
  no.~9, 2877--2885.

\bibitem{Bruggesser/Mani:1971}
H.~Bruggesser and P.~Mani, \emph{Shellable decompositions of cells and
  spheres}, Math. Scand. \textbf{29} (1971), 197--205 (1972).

\bibitem{Chvatal:1974}
Va{\v{s}}ek Chv{\'a}tal, \emph{Intersecting families of edges in hypergraphs
  having the hereditary property}, Hypergraph {S}eminar ({P}roc. {F}irst
  {W}orking {S}em., {O}hio {S}tate {U}niv., {C}olumbus, {O}hio, 1972; dedicated
  to {A}rnold {R}oss), Springer, Berlin, 1974, pp.~61--66. Lecture Notes in
  Math., Vol. 411.

\bibitem{Duval:1996}
Art~M. Duval, \emph{Algebraic shifting and sequentially {C}ohen-{M}acaulay
  simplicial complexes}, Electron. J. Combin. \textbf{3} (1996), no.~1,
  Research Paper 21, approx.\ 14 pp.\ (electronic).

\bibitem{Eagon/Reiner:1998}
John~A. Eagon and Victor Reiner, \emph{Resolutions of {S}tanley-{R}eisner rings
  and {A}lexander duality}, J. Pure Appl. Algebra \textbf{130} (1998), no.~3,
  265--275.

\bibitem{Erdos/Ko/Rado:1961}
P.~Erd{\H{o}}s, Chao Ko, and R.~Rado, \emph{Intersection theorems for systems
  of finite sets}, Quart. J. Math. Oxford Ser. (2) \textbf{12} (1961),
  313--320.

\bibitem{Francisco/VanTuyl:2007}
Christopher~A. Francisco and Adam Van~Tuyl, \emph{Sequentially
  {C}ohen-{M}acaulay edge ideals}, Proc. Amer. Math. Soc. \textbf{135} (2007),
  no.~8, 2327--2337 (electronic), \mbox{arXiv:math/0511022}.

\bibitem{Frankl:1987}
Peter Frankl, \emph{The shifting technique in extremal set theory}, Surveys in
  combinatorics 1987 ({N}ew {C}ross, 1987), London Math. Soc. Lecture Note
  Ser., vol. 123, Cambridge Univ. Press, Cambridge, 1987, pp.~81--110.

\bibitem{Hilton/Holroyd/Spencer:2010UNP}
A.~J.~W. Hilton, F.~C. Holroyd, and C.~L. Spencer, \emph{King {A}rthur and his
  knights with two round tables}, to appear in Q. J. Math.

\bibitem{Hilton/Milner:1967}
A.~J.~W. Hilton and E.~C. Milner, \emph{Some intersection theorems for systems
  of finite sets}, Quart. J. Math. Oxford Ser. (2) \textbf{18} (1967),
  369--384.

\bibitem{Hilton/Spencer:2007}
A.~J.~W. Hilton and C.~L. Spencer, \emph{A graph-theoretical generalization of
  {B}erge's analogue of the {E}rd{\H o}s-{K}o-{R}ado theorem}, Graph theory in
  {P}aris, Trends Math., Birkh\"auser, Basel, 2007, pp.~225--242.

\bibitem{Holroyd/Spencer/Talbot:2005}
Fred Holroyd, Claire Spencer, and John Talbot, \emph{Compression and {E}rd{\H
  o}s-{K}o-{R}ado graphs}, Discrete Math. \textbf{293} (2005), no.~1-3,
  155--164.

\bibitem{Holroyd/Talbot:2005}
Fred Holroyd and John Talbot, \emph{Graphs with the {E}rd{\H o}s-{K}o-{R}ado
  property}, Discrete Math. \textbf{293} (2005), no.~1-3, 165--176,
  \mbox{arXiv:math/0307073}.

\bibitem{Hurlbert/Kamat:2009UNP}
Glenn Hurlbert and Vikram Kamat, \emph{{E}rd{\H o}s-{K}o-{R}ado theorems for
  chordal and bipartite graphs}, to appear in J. Combin. Theory Ser. A,
  \mbox{arXiv:0903.4203}.

\bibitem{Jonsson:2008}
Jakob Jonsson, \emph{Simplicial complexes of graphs}, Lecture Notes in
  Mathematics, vol. 1928, Springer-Verlag, Berlin, 2008.

\bibitem{Kalai:1983}
Gil Kalai, \emph{Enumeration of {${\bf Q}$}-acyclic simplicial complexes},
  Israel J. Math. \textbf{45} (1983), no.~4, 337--351.

\bibitem{Kalai:2002}
\bysame, \emph{Algebraic shifting}, Computational commutative algebra and
  combinatorics ({O}saka, 1999), Adv. Stud. Pure Math., vol.~33, Math. Soc.
  Japan, Tokyo, 2002, pp.~121--163.

\bibitem{Klivans:2007}
Caroline~J. Klivans, \emph{Threshold graphs, shifted complexes, and graphical
  complexes}, Discrete Math. \textbf{307} (2007), no.~21, 2591--2597,
  \mbox{arXiv:math/0703114}.

\bibitem{Mahadev/Peled:1995}
N.~V.~R. Mahadev and U.~N. Peled, \emph{Threshold graphs and related topics},
  Annals of Discrete Mathematics, vol.~56, North-Holland Publishing Co.,
  Amsterdam, 1995.

\bibitem{Nevo:2005}
Eran Nevo, \emph{Algebraic shifting and basic constructions on simplicial
  complexes}, J. Algebraic Combin. \textbf{22} (2005), no.~4, 411--433,
  \mbox{arXiv:math/0303233}.

\bibitem{Smith:1990}
Dean~E. Smith, \emph{On the {C}ohen-{M}acaulay property in commutative algebra
  and simplicial topology}, Pacific J. Math. \textbf{141} (1990), no.~1,
  165--196.

\bibitem{Snevily:1992}
Hunter Snevily, \emph{A new result on {C}hv\'atal's conjecture}, J. Combin.
  Theory Ser. A \textbf{61} (1992), no.~1, 137--141.

\bibitem{Swartz:2006}
Ed~Swartz, \emph{{$g$}-elements, finite buildings and higher {C}ohen-{M}acaulay
  connectivity}, J. Combin. Theory Ser. A \textbf{113} (2006), no.~7,
  1305--1320, \mbox{arXiv:math/0512086}.

\bibitem{Talbot:2003}
John Talbot, \emph{Intersecting families of separated sets}, J. London Math.
  Soc. (2) \textbf{68} (2003), no.~1, 37--51.

\bibitem{VanTuyl/Villarreal:2008}
Adam Van~Tuyl and Rafael~H. Villarreal, \emph{Shellable graphs and sequentially
  {C}ohen-{M}acaulay bipartite graphs}, J. Combin. Theory Ser. A \textbf{115}
  (2008), no.~5, 799--814, \mbox{arXiv:math/0701296}.

\bibitem{Woodroofe:2009a}
Russ Woodroofe, \emph{Vertex decomposable graphs and obstructions to
  shellability}, Proc. Amer. Math. Soc. \textbf{137} (2009), no.~10,
  3235--3246, \mbox{arXiv:0810.0311}.

\end{thebibliography}

\end{document}